\documentclass[12pt,twoside]{article}
\usepackage{amsfonts}

\usepackage{amsmath,mathrsfs,amsfonts,amsthm,color}
\usepackage{graphicx,amssymb,latexsym,a4}

\newtheorem{thm}{Theorem}[section]
\newtheorem{lem}[thm]{Lemma}
\newtheorem{cor}[thm]{Corollary}

\newtheorem{prop}[thm]{Proposition}
\newtheorem{re}[thm]{Remark}

\begin{document}

\title{\bf On a Problem of Harary and Schwenk on Graphs
with Distinct Eigenvalues\footnote{Supported by NSFC No. 11101232
and 11371205, and China Postdoctoral Science Foundation.}}

\author{
{\small Xueliang Li$^a$, \ Jianfeng Wang$^{a,b,}$\footnote{
Corresponding author. \newline{\it \hspace*{5mm}Email addresses:}
lxl@nankai.edu.cn (X. Li), jfwang4@aliyun.com (J. Wang),
 huangqx@xju.edu.cn (Q. Huang)}, \ Qiongxiang Huang$^c$}\\[2mm]
\footnotesize $^a$Center for Combinatorics and LPMC-TJKLC,
Nankai University 300071, Tianjin, Chinaa\\
\footnotesize $^b$Department of Mathematics, Qinghai Normal University,
Xining, Qinghai 810008, China\\
\footnotesize $^c$College of Mathematics and System Science,
Xinjiang University, Urumqi 830046, China}
\date{ }

% Authors and running title to go on top of each page
\pagestyle{myheadings} \markboth{X. Li, J. Wang and Q. Huang}{A
problem of Harary and Schwenk on graphs with distinct eigenvalues}

\maketitle

\begin{abstract}
Harary and Schwenk posed the problem forty years ago: Which graphs
have distinct adjacency eigenvalues? In this paper, we obtain a
necessary and sufficient condition for an Hermitian matrix with
simple spectral radius and distinct eigenvalues. As its application,
we give an algebraic characterization to the Harary-Schwenk's
problem. As an extension of their problem, we also obtain a
necessary and sufficient condition for a positive semidefinite
matrix with simple least eigenvalue and distinct eigenvalues, which
can provide an algebraic characterization to their problem with
respect to the (normalized) Laplacian matrix.

\begin{flushleft}
\end{flushleft}

\begin{flushleft}
\end{flushleft}
\end{abstract}

\section {Introduction}

Let $\mathcal{M}_{n}(\mathbb{F})$ be the set of $n$-by-$n$ matrices
with entries from a field $\mathbb{F}$, where $\mathbb{F}$ is
usually the real number field $\mathbb{R}$ or the complex number
field $\mathbb{C}$ in this paper. By ${\bf C}^n$ and ${\bf R}^n$ we
denote the $n$ dimensional {\it real vector space} and the {\it
complex vector space}, respectively. A matrix $H = [h_{ij}] \in
\mathcal{M}_{n}(\mathbb{C})$ is said to be {\it Hermitian} if
$H=H^\ast$, where $H^\ast = \overline{H}^T = [\overline{h}_{ji}]$.
It is well-known that all the eigenvalues of $H$ are real. So its
eigenvalues can be ordered as $\lambda_1 \geq \lambda_2 \cdots \geq
\lambda_n$, where the set $\sigma(H)$ of them and the largest one
$\lambda_1$ of them are, respectively, called the {\it spectrum} and
the {\it spectral radius} of $H$. Let ${\rm alg}(\lambda)$ and ${\rm
geo}(\lambda)$ be, respectively, the {\it algebraic multiplicity}
and {\it geometric multiplicity} of an eigenvalue $\lambda$. A
well-known fact is that ${\rm geo}(\lambda) \leq {\rm
alg}(\lambda)$. An eigenvalue $\lambda$ is said to be simple if
${\rm alg}(\lambda)=1$.

All graphs considered here are undirected and simple (i.e., loops
and multiple edges are not allowed). Let $G = (V(G), E(G))$ be a
graph with order $n=|V(G)|$ and size $m=|E(G)|$. Let $M=M(G)$ be a
corresponding {\it graph matrix} defined in a prescribed way. The
{\it $M$-eigenvalues} of $G$ are the eigenvalues of $M(G)$. The {\it
$M$-spectral radius} of $G$ is the largest $M$-eigenvalue of $G$. In
the literature there are several graph matrices, including the {\it
adjacency matrix} $A$, the {\it degree matrix $D$}, {\it the
Laplacian matrix} $L=D-A$, the {\it signless Laplacian matrix}
$Q=D+A$ and so on.\smallskip

Generally, most of the $A$-eigenvalues of a graph are distinct. If a
graph has only few distinct $A$-eigenvalues, then it appears that
the graph has a special structure. As an easy example, a connected
graph $G$ has one or two $A$-eigenvalues if and only if  $G$ is,
respectively, an isolated vertex or a complete graph of order at
least two. It has been shown that the strong regular graphs has
three distinct $A$-eigenvalues. This field is perhaps originally
studied by Doob \cite{doob-few}. Subsequently, van Dam makes much
important contributions to this topic \cite{dam-thesis,
dam-non-regular-three}.

On the other side, a graph whose all $A$-eigenvalues are distinct is
a long standing problem that was proposed by Harary and Schwenk
forty years ago (\cite{harary-schwenk},
see also \cite{cve1}, pp. 266):\\[-8mm]

\begin{flushleft}{\bf Harary-Schwenk Problem:}
Which graphs have distinct $A$-eigenvalues?
\end{flushleft}

\vspace{-4mm}

AS far as we know, for forty years more, there have been only two
results on this problem. The first one is due to Mowshowitz
\cite{mowshowitz}.

\begin{prop} \cite{mowshowitz}.
Let $G$ be a finite, simple and undirected graph and $G(X)$ be its
group of automorphisms. If $G$ has distinct $A$-eigenvalues, then
every nonidentity element in $G(X)$ is of order 2 (which implies
$G(X)$ is Abelian).
\end{prop}

The other one is a generalization of the above theorem, due to Chao
\cite{chao}.

\begin{prop} \cite{chao}.
Let $G$ be a finite and simple graph (directed or undirected, with
or without loops) and $G(X)$ be its group of automorphisms. If the
$A$-eigenvalues of $G$ (in the complex number field) are distinct,
then $G(X)$ is Abelian.
\end{prop}

Both of the above two results are related to group theory. In this
paper, we will, however, use the theory of Hermitian matrices to
investigate the Harary-Schwenk Problem. The organization of the
paper is as follows: In Section 2 we study the Hermitian matrices
with simple spectral radius and distinct eigenvalues. In Section 3
we give a complete algebraic characterization for the Harary-Schwenk
Problem. In Section 4 we discuss the positive semidefinite matrix
with simple least eigenvalue and distinct eigenvalues, and extend
the Harary-Schwenk Problem to the (normalized) Laplacian matrix.

\section{Hermitian matrices with simple spectral radius}

First of all, let us recall some important properties of Hermitian
matrices. To start with, the following one is simple but helpful
(see \cite{horn-johnson-book}, Theorem 1.1.6).

\begin{prop}\label{Blambda-fBlambda}
Let $B \in \mathcal{M}_n(\mathbb{F})$ and $g(\cdot)$ be a given
polynomial. If $\alpha$ is an eigenvector of $B$ associated with
$\lambda$, then $\alpha$ is an eigenvector of $g(B)$ associated with
$g(\lambda)$.
\end{prop}

As pointed out in \cite{horn-johnson-book}, provided that $H$ and
$N$ are Hermitian, $H^k$ ($k=1,2,3,\cdots$) and $aH+bN$ ($a,b \in
\mathbb{R}$) are also Hermitian, then the result below is implied.

\begin{prop}\label{BH-fBH}
Let $H \in \mathcal{M}_n(\mathbb{C})$ be an Hermitian matrix and
$g(\cdot)$ be a real polynomial. Then $g(H)$ is Hermitian.
\end{prop}

The next one is the {\it spectral theorem for Hermitian matrices}
\cite{horn-johnson-book}.

\begin{prop}
Let $H \in \mathcal{M}_{n}(\mathbb{\mathbb{C}})$ be given. Then $H$
is Hermitian if and only if there is a unitary matrix $U\in
\mathcal{M}_{n}(\mathbb{\mathbb{C}})$ and a real diagonal matrix
$\Lambda \in \mathcal{M}_{n}(\mathbb{\mathbb{C}})$ such that $H =
U\Lambda U^\ast$. Moveover, $H$ is real Hermitian (i.e., real
symmetric) if and only if there is a real orthogonal matrix $B \in
M_n$ and a real diagonal matrix $\Lambda \in
\mathcal{M}_{n}(\mathbb{\mathbb{C}})$ such that $A = B\Lambda
B^\ast$.
\end{prop}

The above proposition shows that if $H$ is Hermitian or real
symmetric, then $H$ is diagonalizable.  For the diagonalization
matrices, we have the following judgements from the matrix theory
(see, for example, \cite{horn-johnson-book}).

\begin{prop}\label{matrix-theory-diag}
Let $B \in M_n(\mathbb{C})$ and $g(\cdot)$ be a  polynomial.\\[-8mm]
\begin{enumerate}
\item[$\mathrm{(i)}$]
If $B$ is diagonalizable, then $g(B)$ is also diagonalizable;\\[-6mm]
\item[$\mathrm{(ii)}$]
$B$ is diagonalizable if and only if the eigenvalues of $B$ are in
$\mathbb{C}$, and for each eigenvalue $\lambda$ of $B$, ${\rm
alg(\lambda)} = {\rm geo}(\lambda)$.\\[-6mm]
\item[$\mathrm{(iii)}$]
Let all the distinct eigenvalues of $B$ be
$\lambda_1,\lambda_2,\cdots,\lambda_k$. Then $B$ is diagonalizable
if and only if its mimimal polynomial
$m(x)=(x-\lambda_1)(x-\lambda_2)\cdots(x-\lambda_k)$.
\end{enumerate}
\end{prop}

As we know, the rank of an Hermitian $H$ is the number of non-zero
eigenvalues of $H$.  While if the rank is one, the matrix can be
expressed as the following form \cite{horn-johnson-book}.

\begin{prop}\label{matrix-rank-one}
Let $B \in M_n(\mathbb{C})$. Then $B$ has rank one if and only if
there exist two non-zero $n$-vectors ${\bf x}, {\bf y} \in {\bf
C}^n$ such that $B = {\bf x}{\bf y}^*$. Moreover, $B{\bf x} = ({\bf
y}^*{\bf x}){\bf x}$, where ${\bf y}^* = \overline{{\bf y}}^T$.
\end{prop}

Let $I$ and $O$ be the {\it identity matrix} and the {\it zero
matrix} in $\mathcal{M}_{n}(\mathbb{F})$, respectively. For a vector
$\alpha \in {\bf C}^n$, let $\|\alpha\|_2$ be the Euclidean norm,
that is, $\|\alpha\|_2^2 = \alpha^\ast\alpha$. We are now in the
stage to show the following main result of this section.

\begin{thm}\label{C-distinct}
Let $H \in \mathcal{M}_{n}(\mathbb{C})$ be a Hermitian matrix with
simple spectral radius. Then $H$ has exactly $k$ $(2 \leq k \leq n)$
distinct eigenvalues if and only if there are $k$ distinct real
numbers $\lambda_1, \lambda_2, \cdots,\lambda_k$ satisfying\\[-7mm]
\begin{itemize}
\item[$\mathrm{(i)}$]
$H-\lambda_iI$ is a singular matrix for $2 \leq i \leq k$;
\item[$\mathrm{(ii)}$]
$\prod_{i=2}^k(H-\lambda_iI) = b{\bf y}{\bf y}^\ast$ and $H{\bf y} =
\lambda_1{\bf y}$, where $b \in \mathbb{C}\backslash \{0\}$ and
${\bf y} \in {\bf
C}^n \backslash \{\bf 0\}$.\\[-9mm]
\end{itemize}
Moreover, $\lambda_1, \lambda_2, \cdots,\lambda_k$ are exactly the
$k$ distinct eigenvalues of $G$.
\end{thm}

\begin{proof}[\rm \bf Proof]
We first show the necessity. Let $\lambda_1 > \lambda_2 > \cdots >
\lambda_k$ be the $k$ distinct eigenvalues of $H$. Then, $0$ is an
eigenvalue of $H-\lambda_iI$ ($2 \leq i \leq k$) and thus (i)
follows. For (ii), since $\lambda_1$ is simple, then ${\rm
alg}(\lambda_i) = m_i$ $(i=2,3,\cdots,k)$ verifies that
\begin{equation}\label{mi=n-1}
m_2+m_3+\cdots+m_k=n-1.
\end{equation}
Let $f(x) = \prod_{i=2}^k(x-\lambda_i)$. Clearly, $f(x)$ is a real
polynomial. So by Proposition \ref{BH-fBH}, $$f(H)=
\prod_{i=2}^k(H-\lambda_iI)$$ is Hermitian. From Proposition
\ref{Blambda-fBlambda} and \eqref{mi=n-1}, it follows that the
eigenvalues of $f(H)$ are $f(\lambda_1)$ with ${\rm
alg}(f(\lambda_1))=1$ and $f(\lambda_i) = 0$ ($2 \leq i \leq k$)
with ${\rm alg}(0)=n-1$. Hence, the rank of $f(H)$ is one. In line
with Proposition \ref{matrix-rank-one}, there exist two non-zero
$n$-vectors ${\bf x}, {\bf y} \in {\bf C}^n$ such that
\begin{equation}\label{fh=xy}
f(H) = {\bf x}{\bf y}^\ast \quad {\mbox and} \quad f(H){\bf x} =
({\bf y}^\ast{\bf x}){\bf x}.
\end{equation}
The second one of \eqref{fh=xy} indicates that ${\bf y}^\ast{\bf x}$
is just the only one non-zero eigenvalue of $f(H)$. Actually, ${\bf
y}^\ast{\bf x}=f(\lambda_1)$. Due to the first one of \eqref{fh=xy},
we get ${\bf y}^\ast f(H) = f(\lambda_1){\bf y}^\ast$ which leads to
\begin{equation}\label{fhf=flambdaby}
f(H){\bf y} = f(H)^\ast{\bf y} = f(\lambda_1){\bf y}.
\end{equation}
Thereby, both ${\bf x}$ and ${\bf y}$ are the eigenvectors of $f(H)$
associated with eigenvalue $f(\lambda_1)$. By Proposition
\ref{matrix-theory-diag}(i), we obtain that $f(H)$ is
diagonalizable, and thus ${\rm geo}({\bf y}^\ast{\bf x}) = {\rm
alg}({\bf y}^\ast{\bf x}) = 1$ by Proposition
\ref{matrix-theory-diag} (ii). Hence, there exists a non-zero number
$b \in \mathbb{C}$ such that ${\bf x} = b{\bf y}$. Therefore,
\begin{equation}\label{fh=byy}
f(H) = b{\bf y}{\bf y}^\ast, \;\;\mbox{where $b \in
\mathbb{C}\backslash \{0\}$ and ${\bf y} \in {\bf C}^n \backslash
\{\bf 0\}$}.
\end{equation}
In light of Proposition \ref{matrix-theory-diag} (iii), the minimal
polynomial of $H$ is $m(x) = \prod_{i=1}^k(x-\lambda_i)$, and so
$m(H) = \prod_{i=1}^k(H-\lambda_iI) = O$ which together with
\eqref{fh=byy} results in
\begin{equation}\label{hyyy=0}
b(H-\lambda_1I){\bf y}{\bf y}^\ast=O, \;\; \mbox{that is,} \;\;\;
\|{\bf y}\|_2^2(H-\lambda_1I){\bf y}=O.
\end{equation}
Since ${\bf y} \neq O$, then $\|{\bf y}\|_2^2 > 0$ which along with
\eqref{hyyy=0} forces to $(H-\lambda_1I){\bf y}=0$, and consequently
\begin{equation}\label{hy=lambday}
H{\bf y} =\lambda_1{\bf y}.
\end{equation}

For the sufficiency, from (i) it follows that the system of
homogeneous linear equations $(H-\lambda_iI){\bf x}=O$ has a
non-zero solution, say $\alpha_i$, and thus $H\alpha_i =
\lambda_i\alpha_i$ which indicates that $\lambda_i$ is an eigenvalue
of matrix $H$ ($2 \leq i \leq k$). From (ii) we get that $\lambda_1$
is an eigenvalue of $H$. So far, we have shown that $H$ has $k$
distinct eigenvalues $\lambda_1, \lambda_2, \cdots, \lambda_k$.
Assume that $H$ has an extra eigenvalue $\lambda_{k+1}$. Recall that
$f(x) = \prod_{i=2}^k(x-\lambda_i)$, and that $f(\lambda_i)$ ($1\leq
i \leq k+1$) is the eigenvalue of $f(H)$. Obviously, $f(\lambda_1)
\neq 0$, $f(\lambda_i) = 0$ ($2 \leq i \leq k$) and
$f(\lambda_{k+1}) \neq 0$. By (ii) and Proposition
\ref{matrix-rank-one}, the rank of $f(H)$ is one, and so $f(H)$ has
only one none-zero eigenvalue, a contradiction.

This completes the proof.
\end{proof}

Let $\mathbb{R}^+$ be the set of positive real numbers and ${\rm
tr}(B)$ be the trace of matrix $B$.

\begin{cor}\label{R-distinct}
Let $H \in \mathcal{M}_{n}(\mathbb{R})$ be a real symmetric matrix
with simple spectral radius. Then $H$ has exactly $k$ $(2 \leq k
\leq n)$ distinct eigenvalues if and only if there are $k$ distinct
real numbers $\lambda_1, \lambda_2, \cdots,\lambda_k$ satisfying\\[-7mm]
\begin{itemize}
\item[$\mathrm{(i)}$]
$H-\lambda_iI$ is a singular matrix for $2 \leq i \leq k$;
\item[$\mathrm{(ii)}$]
$\prod_{i=2}^k(H-\lambda_iI) = b{\bf y}{\bf y}^T$, where $b \in
\mathbb{R}^+$, $H{\bf y}=\lambda_1{\bf y}$ and ${\bf y} \in {\bf
R}^n \backslash \{\bf 0\}$.\\[-9mm]
\end{itemize}
Moreover, $\lambda_1, \lambda_2, \cdots,\lambda_k$ are exactly the
$k$ distinct $A$-eigenvalues of $G$.
\end{cor}

\begin{proof}[\bf Proof.]
By Theorem \ref{C-distinct} we only need to prove $b \in
\mathbb{R}^+$ and ${\bf y} \in {\bf R}^n$. Employing Theorem
\ref{C-distinct} (i), we get $|H-\lambda_1I| = 0$, and thus
$(H-\lambda_iI){\bf x} = O$ has non-zero real solutions, because $H$
is real symmetric. Hence, ${\bf y} \in {\bf R}^n$. As shown in
Corollary \ref{R-distinct}, ${\bf x} = b{\bf y}$ is the eigenvector
of $f(A)$ associated with eigenvalue $f(\lambda_1) = {\bf y}^T{\bf
x}$. Since $f(H)$ is also real symmetric, then  ${\bf x}, {\bf y}
\in {\bf R}^n$ and thus $b \in \mathbb{R} \backslash \{0\}$. Assume
that $b$ is negative. Then
\begin{equation}\label{fH-R}
f(H) = b{\bf y}{\bf y}^T = -(\sqrt{-b}{\bf y})(\sqrt{-b}{\bf y}^T) =
-\alpha\alpha^T \;\; \mbox{with}\;\; \alpha=\sqrt{-b}{\bf y}.
\end{equation}
Let $\alpha = (a_1,a_2,\cdots,a_n)$ with $a_i \in \mathbb{R}$ ($1
\leq i \leq n$). Hence, by \eqref{fH-R} we get
\begin{equation}\label{trfH<0}
{\rm tr}(f(H))= -\sum_{i=1}^na_i^2 < 0.
\end{equation}
On the other hand, recall that the eigenvalues of $f(H)$ are
$f(\lambda_1) = \prod_{i=1}^2(\lambda_1-\lambda_i)>0$ with ${\rm
alg}(f(\lambda_1)) = 1$ and $0$ with ${\rm alg}(0) = n-1$. Hence,
${\rm tr}(f(H)) = f(\lambda_1)>0$ contradicting to \eqref{trfH<0}.
Therefore, $b \in \mathbb{R}^+$.
\end{proof}

\begin{re}\label{two-methods}
Alternatively, we can adopt a rather different method to prove the
necessity of Theorem \ref{C-distinct}, which will be used in Lemma
\ref{C-S-distinct} in Section 4. But, the reader will see that the
method to show the necessity of Theorem \ref{C-distinct} is not
suitable for Lemma \ref{C-S-distinct}, and that it can provide more
information (see the corollary below).
\end{re}

We now look back on Proposition \ref{Blambda-fBlambda} again. Most
notably, its {\it converse} is not generally true. As an example, if
we set $f(x)=x^3+x^2+6$ and
$$B=\begin{pmatrix}
    1 & 1 & 2\\
    0 & -1 & 0\\
    2 & 0 & -1
    \end{pmatrix}\!\!,
\;\;\mbox{then}\;\; f(B)=\begin{pmatrix}
    16 & 5 & 10\\
    0 & 6 & 0\\
    10 & 0 & 6
    \end{pmatrix}\!\!.
$$
Easily to obtain $\sigma(B)=\{\pm\sqrt{5},-1\}$ and $\sigma(f(B)) =
\{11\pm\sqrt{5},6\}$. Obviously, $f(0)=6 \in \sigma(f(B))$, but $0
\not\in \sigma(B)$. Even so, the derivation from
\eqref{fhf=flambdaby} to \eqref{hy=lambday} of Theorem
\ref{C-distinct} provides a special case which makes the  {\it
converse} true.

\begin{cor}
Let $H \in \mathcal{M}_n(\mathbb{C})$ be an Hermitian matrix with
simple spectral radius $\lambda$ and minimal polynomial $m(x)$. Set
$f(x)=\frac{m(x)}{x-\lambda}$. If $\alpha$ is an eigenvector of
$f(H)$ associated with eigenvalue $f(\lambda)$, then $\alpha$ is an
eigenvector of $H$ associated with $\lambda$.
\end{cor}

\section{Applications to the Harary-Schwenk Problem}

We now apply Corollary \ref{R-distinct} to tackle the Harary-Schwenk
Problem. Clearly, the adjacency matrix $A$ is real symmetric. If $G$
is a connected graph, then $A$ is an irreducible nonnegative matrix.
By Perron-Frobenius Theorem we know that the $A$-spectral radius is
simple and its associated eigenvector is positive. Setting $\alpha =
\sqrt{b}{\bf y}$ in Corollary \ref{R-distinct}, we get the following
result.

\begin{thm}\label{A-distinct}
Let $G$ be connected graph of order $n \geq 2$. Then $G$ has exactly
$k$ $(2 \leq k \leq n)$ distinct $A$-eigenvalues if and only if
there are $k$ distinct real numbers $\lambda_1,
\lambda_2, \cdots,\lambda_k$ satisfying\\[-9mm]
\begin{itemize}
\item[$\mathrm{(i)}$]
$A-\lambda_iI$ is a singular matrix for $2 \leq i \leq k$;
\item[$\mathrm{(ii)}$]
$\prod_{i=2}^k(A-\lambda_iI) = \alpha\alpha^T$ and $A\alpha =
\lambda_1\alpha$, where $\alpha \in {\bf R}^n  \backslash \{\bf 0\}$.\\[-9mm]
\end{itemize}
Moreover, $\lambda_1, \lambda_2, \cdots,\lambda_k$ are exactly the
$k$ distinct $A$-eigenvalues of $G$.
\end{thm}

van Dam \cite{dam-non-regular-three} used the following equality
$$(A-\lambda_2I)(A-\lambda_2) = \alpha\alpha^T, \quad
\mbox{with}\;\; A\alpha = \lambda_1\alpha,$$ to show that a regular
graph with three distinct $A$-eigenvalues must be strongly regular.
While, it is just the case $k=3$ in Theorem \ref{A-distinct}.

\begin{re}
Noteworthily, the case $k=n$ in Theorem \ref{A-distinct} is exactly
the answer to the Harary-Schwenk Problem, which offers an algebraic
characterization. Further investigation is how to determine the
structures of such graphs by the theorem, however, this will not be
as easy as it is seen.
\end{re}

\begin{re}
When $G$ is connected, the signless Laplacian matrix $Q$ has the
same properties as above those of adjacency matrix. So, Theorem
\ref{A-distinct} also holds for the signless Laplacian matrix.
\end{re}

In the end of this section, we give a new proof for the relation
between the diameter and the number of distinct $A$-eigenvalues of a
graph; see \cite{cve1} for example.

\begin{cor}
Let $G$ be a connected graph with $k$ distinct $A$-eigenvalues. Then
 the diameter $\mathrm{diam}(G)$ of $G$ is at most $k-1$.
\end{cor}

\begin{proof}[\rm \bf Proof]
By Theorem  \ref{A-distinct} (ii) we get
\begin{equation}\label{bij}
\prod_{i=2}^k(A-\lambda_iI) =
A^{k-1}+a_1A^{k-2}+a_2A^{k-3}+\cdots+a_{k-2}A+a_{k-1}I
=\alpha\alpha^T = (b_{ij})_{n \times n}.
\end{equation}
Since $\alpha$ is a positive vector, then $b_{ij} > 0$. Assume that
$\mathrm{diam}(G) > k-1$. By the definition of diameter, for some
$v_i$ and $v_j$ the elements $a_{ij}^{(s)}$ from the $i$-th row and
from the $j$-column of the matrices $A^{(s)}$ ($1 \leq s \leq k-1$)
satisfy
$$a_{ij}^{(k-1)}=a_{ij}^{(k-2)}=\cdots=a_{ij}=0,$$
which together with \eqref{bij} results in $$b_{ij}=
a_{ij}^{(k-1)}+\alpha_{1}a_{ij}^{(k-2)}+\alpha_{2}a_{ij}^{(k-3)}
+\cdots+\alpha_{k-2}a_{ij}=0,$$ a contradiction. Hence,
$\mathrm{diam}(G) \leq k-1$.
\end{proof}

\section{Extensions of Harary-Schwenk Problem}

In this section we extend the Harary-Schwenk Problem to other graph
matrices. As is known to all, the positive semidefinite matrices
must be Hermitian, and the eigenvalues of such matrices are
nonnegative real numbers.

\begin{lem}\label{C-S-distinct}
Let $H \in \mathcal{M}_{n}(\mathbb{C})$ be a positive semidefinite
matrix with simple least eigenvalue. Then $H$ has exactly $k$ $(2
\leq k \leq n)$ distinct eigenvalues if and only if there are $k$
distinct real numbers $\mu_1, \mu_2, \cdots, \mu_{k}$ satisfying
\begin{itemize}
\item[$\mathrm{(i)}$]
$H-\mu_iI$ is a singular matrix for $2 \leq i \leq k$;
\item[$\mathrm{(ii)}$]
$\prod_{i=2}^k(H-\mu_iI) =
\frac{\prod_{i=1}^{k-1}(\mu_k-\mu_i)}{\|\alpha\|_2^2}\alpha\alpha^\ast$
and $H\alpha = \mu_k\alpha$, where $\alpha \in {\bf C}^n
\backslash \{\bf 0\}$.\\[-8mm]
\end{itemize}
Moreover, $\mu_1, \mu_2, \cdots,\mu_k$ are exactly the $k$ distinct
eigenvalues of $H$.
\end{lem}

\begin{proof}[\rm \bf Proof]
Let $\mu_1 > \mu_2 > \cdots > \mu_{k-1} > \mu_k$ be the $k$ distinct
eigenvalues of $H$. Since $H$ is diagonalizable, by Proposition
\ref{matrix-theory-diag} (iii) we get that the minimal polynomial of
$H$ is $$m(x)= (x-\mu_k)(x-\mu_1)\cdots(x-\mu_{k-1}),$$ which leads
to
$$(H-\mu_kI)\prod_{i=1}^{k-1}(H-\mu_iI)=0.$$
Let $\alpha = (a_1,a_2,\cdots,a_n)$ be the eigenvector of $H$
associated to eigenvalue $\mu_k$. Since ${\rm alg}(\mu_k) =1$, then
by Proposition \ref{matrix-theory-diag} (ii) we have ${\rm
geo}(\mu_k) = 1$, which indicates that any eigenvector of $H$
associated to the eigenvalue $\mu_k$ is a scalar multiple of
$\alpha$. Hence, each column of matrix $\prod_{i=1}^{k-1}(H-\mu_iI)$
can be written in the form $b_i\alpha$ with $b_i \in \mathbb{C}$
($i=1,2,\cdots,n$), and so
\begin{equation}\label{b1bn}
\prod_{i=1}^{k-1}(H-\mu_iI) = \alpha(b_1,b_2,\cdots,b_n).
\end{equation}
Since $\alpha^\ast(H-\mu_iI) = \alpha^\ast H-\mu_i\alpha^\ast =
(\mu_k-\mu_i)\alpha^\ast$, multiplying $\alpha^\ast$ to both sides
of \eqref{b1bn}, we obtain
$$\prod_{i=1}^{k-1}(\mu_k-\mu_i)\alpha^\ast =
\alpha^\ast\alpha(b_1,b_2,\cdots,b_n) =
\|\alpha\|_2^2(b_1,b_2,\cdots,b_n).$$ Thereby,
$$b_i = \frac{\prod_{i=1}^{k-1}(\mu_k-\mu_i)}{\|\alpha\|_2^2}a_i,
\;\;i=1,2,\cdots,n.$$ Hence, the necessity follows.

The proof of the sufficiency is similar to that of Theorem
\ref{C-distinct}. From (i) it follows that the system of homogeneous
linear equations $(H-\mu_iI){\bf x}=O$ has a non-zero solution, say
$\alpha_i$, and thus $H\alpha_i = \mu_i\alpha_i$ which indicates
that $\mu_i$ is an eigenvalue of matrix $H$ ($2 \leq i \leq k$).
From (ii) we get that $0$ is an eigenvalue of $H$. Therefore, we
have shown that $H$ has $k$ distinct eigenvalues $\mu_1, \mu_2,
\cdots,\mu_{k-1}, 0$. Assume that $H$ has an extra eigenvalue
$\mu_{k+1}$. Recall that $f(x) = \prod_{i=1}^{k-1}(x-\lambda_i)$ and
that $f(\mu_i)$ ($1\leq i \leq k+1$) is the eigenvalue of $f(H)$.
Obviously, $f(\lambda_i) = 0$ ($1 \leq i \leq k-1$), $f(0) \neq 0$
and $f(\mu_{k+1}) \neq 0$. By (ii) and Proposition
\ref{matrix-rank-one}, the rank of $f(H)$ is one, and so $f(H)$ has
only one none-zero simple eigenvalue, a contradiction.

This finishes the proof.
\end{proof}

\begin{re}
Go back to Remark \ref{two-methods}. Applying the method used in
Lemma \ref{C-S-distinct} to\\[2mm] the necessity of Theorem
\ref{C-distinct}, we get
$b=\frac{\prod_{i=2}^k(\lambda_1-\lambda_i)}{\|\alpha\|_2^2}>0$.
\end{re}

The following corollary immediately follows from Lemma
\ref{C-S-distinct}.

\begin{cor}\label{R-S-distinct}
Let $H \in \mathcal{M}_{n}(\mathbb{R})$ be a positive semidefinite
matrix with simple least eigenvalue. Then $H$ has exactly $k$ $(2
\leq k \leq n)$ distinct eigenvalues if and only if there are $k$
distinct real numbers $\mu_1, \mu_2, \cdots, \mu_{k}$ satisfying\\[-7mm]
\begin{itemize}
\item[$\mathrm{(i)}$]
$H-\mu_iI$ is a singular matrix for $2 \leq i \leq k$;
\item[$\mathrm{(ii)}$]
$\prod_{i=2}^k(H-\mu_iI) =
\frac{\prod_{i=1}^{k-1}(\mu_k-\mu_i)}{\|\alpha\|^2}\alpha\alpha^T$
and $H\alpha=\mu_k\alpha$, where $\alpha \in {\bf R}^n  \backslash \{\bf 0\}$.\\[-7mm]
\end{itemize}
Moreover, $\mu_1, \mu_2, \cdots,\mu_k$ are exactly the $k$ distinct
eigenvalues of $H$.
\end{cor}

It is generally known that the Laplacian matrix $L$ of a graph $G$
is real symmetric and positive semidefinite. Moreover,  $0$ is the
least $L$-eigenvalue with eigenvector $\alpha^T=(1,1,\cdots,1) \in
{\bf R}^n$, and ${\rm alg}(0)$ is equal to the number of the
connected components of $G$. So, if $G$ is connected, then $0$ is a
simple $L$-eigenvalue. Consequently, substituting $\mu_k = 0$,
$\alpha\alpha^T = J$ (the {\it all-ones matrix}) and $\|\alpha\|^2 =
n$ into Corollary \ref{R-S-distinct} we get the following result.

\begin{thm}
Let $G$ be a connected graph with order $n$. Then $G$ has $k$ $(2
\leq k \leq n)$ distinct $L$-eigenvalues if and only if there are
$k-1$ distinct non-zero real
numbers $\mu_1, \mu_2, \cdots, \mu_{k-1}$ satisfying\\[-7mm]
\begin{itemize}
\item[$\mathrm{(i)}$]
$L-\mu_iI$ is a singular matrix for $2 \leq i \leq k$;
\item[$\mathrm{(ii)}$]
$\prod_{i=2}^k(L-\mu_iI)
=(-1)^{k-1}\frac{\prod_{i=1}^{k-1}\mu_i}{n}J$, where $J$ is the all-ones matrix.\\[-8mm]
\end{itemize}
Moreover, $\mu_1, \mu_2, \cdots,\mu_{k-1},0$ are exactly the $k$
distinct $L$-eigenvalues of $G$.
\end{thm}

We finally turn to the {\it normalized Laplacian matrix} of a graph
$G$, which is introduced by Chung \cite{chung-book} and defined as
$$\mathcal{L} = D^{-\frac{\;1}{2}}LD^{\frac{\;1}{2}}. $$
Mohar \cite{mohar-some} calls this matrix the {\it transition
Laplacian}. As a fact, $\mathcal{L}$ is also real symmetric and
positive semidefinite. Moreover, if $G$ is connected, $0$ is a
simple least $\mathcal{L}$-eigenvalue with eigenvector $\alpha^T =
(\sqrt{d_1},\sqrt{d_2},\cdots,\sqrt{d_n})$, where $d_i$ is the
degree of vertex $v_i$ ($1 \leq i \leq n$). Substituting $\mu_k = 0$
and $\|\alpha\|^2 = \sum_{i=1}^nd_i$ into Corollary
\ref{R-S-distinct} we obtain the following results.

\begin{thm}
Let $G$ be a connected graph with order $n$ and size $m$. Then $G$
has $k$ distinct $\mathcal{L}$-eigenvalues if and only if there are
$k-1$ $(2 \leq k \leq n)$ distinct non-zero real numbers $\mu_1,
\mu_2,\cdots,\mu_{k-1}$ satisfying\\[-7mm]
\begin{itemize}
\item[$\mathrm{(i)}$]
$\mathcal{L}-\mu_iI$ is a singular matrix for $2 \leq i \leq k$;
\item[$\mathrm{(ii)}$]
$\prod_{i=2}^k(\mathcal{L}-\mu_iI)
=(-1)^{k-1}\frac{\prod_{i=1}^{k-1}\mu_i}{2m}\alpha\alpha^T$, where
$\alpha^T=(\sqrt{d_1},\sqrt{d_2},\cdots,\sqrt{d_n})$.\\[-8mm]
\end{itemize}
Moreover, $\mu_1, \mu_2, \cdots,\mu_{k-1},0$ are exactly the $k$
distinct $\mathcal{L}$-eigenvalues of $G$.
\end{thm}

\end{document}